\documentclass{article}

\usepackage{graphicx}
\usepackage{amsmath, amsthm,pgf,tikz}
\usetikzlibrary{arrows}
\usepackage{amsfonts}
\usepackage{amssymb}
\usepackage{url}
\usepackage{hyperref}  
\usepackage{verbatim}
\usepackage{color}
\definecolor{red}{rgb}{1,0,0}
\def\red{\color{red}}
\definecolor{blu}{rgb}{0,0,1}

\usepackage[mathlines]{lineno}

\usepackage{pgf,tikz}
\usepackage{mathrsfs}
\usetikzlibrary{arrows}
\definecolor{qqqqff}{rgb}{0.,0.,1.}

\def\noi{\noindent}

\newcommand{\Gc}{\overline{G}}

\newcommand{\pd}{\gamma_P}

\newcommand{\Z}{\operatorname{Z}}
\newcommand{\maxpd}{\mathcal{T}}
\newcommand{\pbar}{\bar p}

\newcommand{\tw}{\operatorname{tw}}

\newcommand{\diam}{\operatorname{diam}}

\newtheorem{thm}{Theorem}[section]
\newtheorem{cor}[thm]{Corollary}
\newtheorem{lem}[thm]{Lemma}
\newtheorem{prop}[thm]{Proposition}
\newtheorem{conj}[thm]{Conjecture}
\newtheorem{obs}[thm]{Observation}

\newtheorem{quest}[thm]{Question}

\theoremstyle{definition}
\newtheorem{rem}[thm]{Remark}

\theoremstyle{definition}

\theoremstyle{definition}
\newtheorem{ex}[thm]{Example}

\newcommand{\bit}{\begin{itemize}}
\newcommand{\eit}{\end{itemize}}
\newcommand{\ben}{\begin{enumerate}}
\newcommand{\een}{\end{enumerate}}
\newcommand{\beq}{\begin{equation}}
\newcommand{\eeq}{\end{equation}}
\newcommand{\bea}{\begin{eqnarray*}}
\newcommand{\eea}{\end{eqnarray*}}
\newcommand{\bpf}{\begin{proof}}
\newcommand{\epf}{\end{proof}}

\newcommand{\lc}{\left\lceil}
\newcommand{\rc}{\right\rceil}
\newcommand{\lf}{\left\lfloor}
\newcommand{\rf}{\right\rfloor}

\title{Note on Nordhaus-Gaddum problems for power domination}

\author{Katherine F. Benson\thanks{Department of Mathematics and Physics, Westminster College, Fulton, MO 65251, USA (katie.benson@westminster-mo.edu).} \and Daniela Ferrero\thanks{Department of Mathematics, Texas State University, San Marcos, TX 78666, USA (dferrero@txstate.edu).} \and Mary Flagg\thanks{Department of Mathematics, Computer Science and Cooperative Engineering, University of St. Thomas, 3800 Montrose, Houston, TX 77006, USA (flaggm@stthom.edu).} \and Veronika Furst\thanks{Department of Mathematics, Fort Lewis College, Durango, CO 81301,
USA (furst$\_$v@fortlewis.edu).} \and Leslie Hogben\thanks{Department of Mathematics, Iowa State University,
Ames, IA 50011, USA (LHogben@iastate.edu) and American Institute of Mathematics, 600 E. Brokaw Road, San Jose, CA 95112, USA
(hogben@aimath.org).} \and Violeta Vasilevska\thanks{Department of Mathematics, Utah Valley University, Orem, UT, 84058, USA (Violeta.Vasilevska@uvu.edu).} }

\begin{document}
\maketitle

\vspace{-10pt}
\begin{abstract} The upper and lower Nordhaus-Gaddum bounds over all graphs for the power domination number follow from known bounds on the domination number and examples.  In this note we improve the upper sum bound for the power domination number substantially for graphs having the property that both the graph and its complement must be connected.  For these graphs, our bound is tight and is also significantly better than the corresponding bound for domination number.   
We also   improve the product upper bound for the power domination number  for graphs with certain properties.  \end{abstract}

\noi {\bf Keywords} power domination, domination, zero forcing, Nordhaus-Gaddum

\noi{\bf AMS subject classification} 05C69, 05C57 


\section{Introduction}\label{sintro}
 
The study of the power domination number of a graph arose from the question of how to monitor electric power networks at minimum cost, see    
Haynes et al. \cite{HHHH02}. Intuitively, the power domination problem consists of finding a set of vertices in a graph that can observe the entire graph according to certain observation rules. The formal definition is given below immediately after some graph theory terminology.

A  {\em graph} $G=(V,E)$  is an ordered pair formed by a finite nonempty set of {\em  vertices}
$V=V(G)$ and a set of {\em edges} $E=E(G)$ containing unordered pairs
of distinct vertices (that is, all graphs are simple and undirected). 
The {\em complement} of  $G=(V,E)$ is the 
graph $\overline{G}=(V,\overline{E})$, where $\overline{E}$
consists of all two element subsets of $V$ that are not in $E$.  For any vertex $v \in V$, the {\em neighborhood} of $v$ is the set $N(v) = \{u \in V: \{u,v\} \in E
\}$ 
and the
{\em closed neighborhood} of $v$ is the set $N[v] = N(v) \cup \{ v \}$.
Similarly, for any set of vertices $S$, $N(S) = \cup_{v \in S} N(v)$
and $N[S] = \cup_{v \in  S} N[v]$.  

 For a  set $S$ of vertices in a graph $G$, define  $PD(S)\subseteq V(G)$ recursively as follows:
 \ben
 \item $PD(S):=N[S]= S\cup N(S)$.
 \item While there exists $ v\in PD(S)$ such that $|N(v)\setminus PD(S)|=1$:  $PD(S):=PD(S)\cup N(v)$.
 \een
 A set $S\subseteq V(G)$ is called a {\em power dominating set} of  a graph $G$ if, at the end of the process above, $PD(S)=V(G)$. A {\em minimum power dominating set} is a power dominating set of minimum cardinality.  The {\em power domination number} of  $G$,  denoted by $\pd (G)$,  is the cardinality of a minimum power dominating set.

Power domination is naturally related to domination and to  zero forcing. A set $S\subseteq V(G)$ is called a {\em dominating set} of  a graph $G$ if $N[S]=V(G)$. A {\em minimum dominating set} is a dominating set of minimum cardinality.  The {\em domination number}  of  $G$, denoted by $\gamma(G)$, is the cardinality of a minimum dominating set.  Clearly $\pd(G)\le \gamma(G)$.  

Zero forcing was introduced independently in combinatorial matrix theory \cite{AIM08} and control of quantum systems \cite{BG07}.  
From a graph theory point of view, zero forcing is a coloring game on a  graph played according to the {\em color change rule}: If $u$ is a blue vertex and exactly one  neighbor $w$ of $u$ is white, then change the color of $w$ to blue.  We say $u$ {\em forces} $w$. A {\em zero forcing set} for  $G$ is a subset of vertices $B$ such that when the vertices in $B$ are colored blue and the remaining vertices are colored white initially,  repeated application of the color change rule can color all vertices of $G$ blue.  A {\em minimum zero forcing set} is a zero forcing set of minimum cardinality.  The {\em zero forcing number}  of  $G$, denoted by $\Z(G)$, is the cardinality of a minimum zero forcing set.   
Power domination can be seen as a domination step followed by a zero forcing process, and we will use the terminology ``$v$ forces $w$'' to refer to Step 2 of power domination. Clearly $\pd(G)\le \Z(G)$. 

For a graph parameter $\zeta$, the following are \emph{Nordhaus-Gaddum} problems:
\bit
\item Determine a (tight) lower or upper bound on $\zeta(G)+\zeta(\Gc)$.
\item Determine a (tight) lower or upper bound on $\zeta(G)\cdot \zeta(\Gc)$.
\eit
The name comes from the next theorem of Nordhaus and Gaddum, where $\chi(G)$ denotes the chromatic number  of $G$.
\begin{thm}\label{NGthm} {\rm \cite{NG}}  For any graph $G$ of order $n$, \vspace{-3pt}
\[2\sqrt{n}\le \chi(G)+\chi(\Gc)\le n+1\vspace{-3pt}\]
and\vspace{-3pt}
\[n\le \chi(G)\cdot\chi(\Gc)\le \left(\frac{ n+1} 2\right)^2.\vspace{-3pt}\]
Each bound is assumed for infinitely many values of $n$.
\end{thm}

Nordhaus-Gaddum bounds have been found for both domination and zero forcing.  In addition to the original papers cited here, Nordhaus-Gaddum results for domination and several variants (but not power domination) are discussed in Section 9.1 of the book  \cite{HHS98} and in the survey paper \cite{NGsurvey13}.  
\begin{thm}{\rm \cite{JP72}}\label{NGdom}
For any graph $G$ of order $n\ge 2$,
\[3\le \gamma(G) +\gamma(\Gc)\le n + 1 \qquad\mbox{ and }\qquad 2\le \gamma(G) \cdot\gamma(\Gc)\le n.\]
The upper bounds are realized by the complete graph $K_n$, and the lower bounds are realized by the star (complete bipartite graph) $K_{1,n-1}$.
\end{thm}

It is  known  that for a graph $G$ of order $n\ge 2$,\vspace{-3pt}
\[n-2\le \Z(G)+\Z(\Gc)\le 2n-1\vspace{-3pt}\]
 and\vspace{-3pt}
\[n-3\le \Z(G)\cdot\Z(\Gc)\le n^2-n,\vspace{-3pt}\]
with the upper bounds realized by the complete graph $K_n$ and the lower bounds realized by the path $P_n$ for $n\ge 4$.  That the  upper bounds are correct is immediate.  The result $n-2\le \Z(G)+\Z(\Gc)$ appears in  \cite{EGR11}.  Then $n-3\le \Z(G)\cdot\Z(\Gc)$ follows, because $1\le \Z(G)$ for all $G$ and the function $f(z)=z(n-2-z)$ attains its minimum on the interval $[1,n-3]$ at the endpoints.   

The general Nordhaus-Gaddum upper bounds for power domination number follow from those for domination number given in  
Theorem \ref{NGdom}.
The  inequalities $2\le \pd(G) +\pd(\Gc)$ and $1\le \pd(G) \cdot\pd(\Gc)$ are obvious since $1\le \pd(G)$ for every graph, and these are realized by the path $P_n$ (it is straightforward to verify  that $\pd(P_n)=1=\pd(\overline{P_n})$). 
\begin{cor}\label{NGpowerdom}
For any graph $G$  of order $n$,\vspace{-3pt}
\[2\le \pd(G) +\pd(\Gc)\le n + 1 \qquad\mbox{ and }\qquad1\le \pd(G) \cdot\pd(\Gc)\le n.\vspace{-3pt}\]
The upper bounds are realized by the complete graph $K_n$, and the lower bounds are realized by the path $P_n$. 
\end{cor}

In Section \ref{sNGsum} we improve  the sum upper bound for the power domination number  significantly under the assumption that  both $G$ and $\Gc$ are connected, or more generally all components of  both have order at least 3, and show that this bound is substantially different from the analogous bound for domination number.  In Section \ref{sNGprod} we refine the product bounds for certain special cases.   Section \ref{stools} contains additional results that   we  use in Sections \ref{sNGsum} and \ref{sNGprod}.  Section 5 summarizes the bounds for  domination number, power domination number, and zero forcing number. 

Some additional notation is used: 
Let  $K_{p,q}$ denote a complete bipartite graph with partite sets of cardinality $p$ and $q$.  The {\em degree} of vertex $v$ is $\deg_G v=|N_G(v)|$. Let $\delta(G)$ (respectively, $\Delta(G)$)  denote the minimum (respectively, maximum) of the degrees of the vertices of $G$.  A {\em cut-set} is a set of vertices whose removal disconnects $G$. The {\em vertex-connectivity} of $G\ne K_n$, denoted by $\kappa (G)$, is  the minimum cardinality of  a cut-set  (note $\kappa(G)=0$ if $G$ is disconnected), and  $\kappa(K_n)=n-1$. An {\em edge-cut} is a set of edges whose removal disconnects $G$, and the {\em edge-connectivity} of $G$, denoted by  $\lambda (G)$, is the minimum cardinality of an  edge-cut. Observe that $\kappa (G)\leq \lambda (G) \leq \delta (G)$.  The {\em distance}  between vertices $u$ and $v$ in $G$, $d_G(u,v)$,  is the length of a shortest path between $u$ and $v$ in $G$.  The {\em diameter} of $G$, $\diam (G)$, is the maximum distance between  two vertices in a connected graph $G$; $\diam(G)=\infty$ if $G$ is not connected.  A {\em component} of a graph is a maximal connected subgraph. 


\section{Tools for Nordhaus-Gaddum bounds for power domination}\label{stools}

In this section we establish results that will be applied to improve Nordhaus-Gaddum upper bounds for both the sum and product of the power domination number with additional assumptions, such as every component of the graph and its complement  has order at least 3.  
The next result is immediate from Corollary \ref{NGpowerdom}.

\begin{cor}\label{cor:NGpowerdom} For any graph $G$  of order $n$, $\pd(G)\le \lf \frac n {\pd(\Gc)}\rf$.
\end{cor}

Next we consider the relationship between the power domination number of $G$ or $\Gc$ and the minimum degree or vertex-connectivity of $G$.

 \begin{rem} For any graph $G$ of order $n$, $\gamma(\Gc)\le\delta(G)+1$, because a vertex of maximum degree in $\Gc$, which is $n-1-\delta(G)$, together with all its non-neighbors is a dominating set of $\Gc$.
\end{rem}

\begin{prop}\label{GBarBd1} 
Let $G$ be a graph such that neither $G$ nor $\Gc$ has  isolated vertices.
Then $\pd(\Gc)\le\delta(G)$.  If $\delta(G)=1$, then $\pd(\Gc)=1$.
\end{prop}
\begin{proof}
Construct a power dominating set $S$ for $\Gc$ of cardinality $\delta(G)$ as follows: Put a vertex $v$ of maximum degree in $\Gc$ into $S$, so $|N_{\Gc}[v]|=\Delta(\Gc)+1=n-1-\delta(G)+1=n-\delta(G)<n$, where $n$ is the order of $G$.  Then add all but one of the vertices in $V(\Gc)\setminus N_{\Gc}[v]$ into $S$, i.e., add $\delta(G)-1\ge 0$ vertices to $S$, so $|S|=\delta(G)$.  Now $N_{\Gc}[S]$ contains all but at most one vertex, and since $\Gc$ has no isolated vertices, any neighbor of such a vertex can force it.  The last statement then follows since $\pd(G)\ge 1$ for all graphs $G$.
\end{proof}

\begin{thm}\label{kappa}{\rm \cite{HV06}} If $G$ is a graph with $\diam (G) =2$, then $\gamma (G)\leq \kappa(G)$. 
\end{thm}\vspace{-2pt}

Next we state several results that give sufficient conditions for $\gamma(G)\le 2$ or $\gamma(\Gc)\le 2$, which then imply $\pd(G)\le 2$ or $\pd(\Gc)\le 2$.  

\begin{thm}\label{diam3+} {\rm \cite{BCD88}, \cite[Theorem 2.25]{HHS98}}
If $G$ is a  graph with $\diam(G)\geq 3$, then $\gamma (\Gc)\leq 2$. 
\end{thm}\vspace{-2pt}
Note that Theorem \ref{diam3+} also applies to graphs that are not connected.

\begin{thm}\label{kappapd}
Suppose  $G$ is a graph with $\diam (G)=2$ such that $\Gc$ has no isolated vertices.  
 Then $\pd (G)\leq  \kappa (G)-1$ or $\pd (\Gc)\leq 2$. 
\end{thm} \vspace{-12pt}

\begin{proof} Since $\Gc$ has no isolated vertices, every vertex has a neighbor  in $\Gc$.  Let $S$ be a minimum cut-set for $G$. Since $\diam (G)=2$, every vertex in $V\setminus S$ is adjacent to at least one  vertex in $S$. 

 \underline{Case 1:} There exists  a vertex $u\in V\setminus S$ that is adjacent to exactly one vertex in $S$, say $v$ (Case 1 is the only possible case when $\kappa(G)=1$). Let $G_1$ denote the component of $G-S$ containing $u$.  
  In $\Gc$, $u$ dominates $S\setminus \{v\}$ and all vertices in components of $G-S$ other than $G_1$. Let $x$ be any vertex in a component of $G-S$ that is not equal to $G_1$.  Then $x$ dominates the vertices of $G_1$. Therefore,  $\{u,x\}$ dominates all vertices in $V$ except possibly $v$, and any neighbor of $v$ in $\Gc$ can force $v$, so $\{u,x\}$ is a power dominating set for $\Gc$. Thus, $\pd (\Gc)\leq  2$.
  
 \underline{Case 2:} Every vertex in $V\setminus S$ is adjacent to  at least two vertices in $S$.  Then $S\setminus \{v\}$ is a power dominating set for any vertex $v\in S$, because $S\setminus \{v\}$ dominates $V\setminus \{v\}$,  
and any neighbor of $v$ in $G$ can force $v$. Thus, $\pd (G)\leq  \kappa (G)-1$.
\end{proof}

\begin{thm}{\rm \cite{GH02}}\label{planar2}
If $G$ is planar and $\diam(G)=2$, then $\gamma(G)\le 2$ or $G=S_4(K_3)$, the graph shown in Figure \ref{fig:S4K3}.  Furthermore, $\gamma(S_4(K_3))=3$.
\end{thm}\vspace{-10pt}

\begin{figure}[!ht]
\begin{center}
\scalebox{.3}{\includegraphics{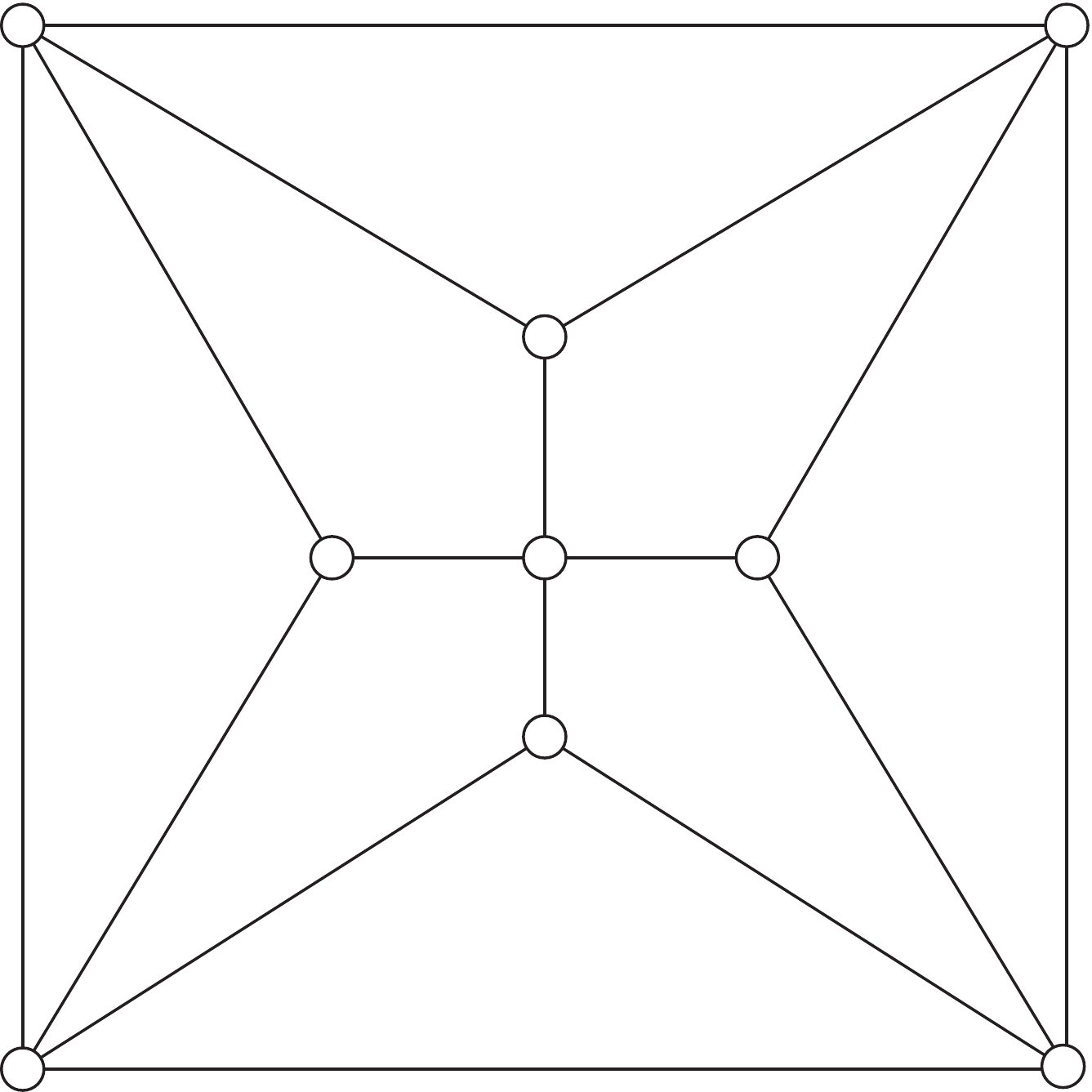}}\vspace{-4pt}
\caption{The graph $S_4(K_3)$, which is the only planar graph with diameter 2 and domination number greater than $2$.}\label{fig:S4K3}\vspace{-10pt}
 \end{center}
\end{figure}

\begin{cor}\label{planar2pd} If $G$ is planar and $\diam(G)=2$, then $\pd(G)\le 2$. \end{cor} 
\bpf This follows from Theorem \ref{planar2} and the fact that $\pd(S_4(K_3))=2$. \epf

 When $\kappa (G)=  \lambda (G)=\delta (G)$, $G$ is {\em maximally connected}. In every maximally connected graph $G$, for any vertex $v$ such that $\deg v =\delta (G)$, $N_G(v)$ is a minimum cut-set and the set of all edges incident with $v$ is a minimum edge-cut. In this case we say the cut is {\em trivial}, because it leaves a connected component formed by one isolated vertex. A maximally connected graph $G$ is {\em super-$\lambda$} if every minimum edge-cut is trivial. Super-$\lambda$ graphs of diameter $2$ were characterized by Wang and Li:

\begin{thm} {\rm \cite{WL99}} A connected graph $G$ with $\diam (G)=2$ is super-$\lambda$ if and only if $G$ contains no subgraph $K_{\delta (G)}$ in which all of the vertices have degree $($in $G)$ equal to $\delta (G)$.
\end{thm}

\begin{prop}\label{superlambda}
Let $G$ be a connected graph with $\diam (G)=2$. If $G$ is not super-$\lambda$, then $\pd (G)\leq 2$.
\end{prop}

\bpf
Since  $G$ is not super-$\lambda$, there exists a subgraph  $K_{\delta (G)}$ in which all of the vertices have degree equal to $\delta (G)$ in $G$. Let $v$ be a vertex in this $K_{\delta (G)}$, so $v$  has exactly one neighbor outside $K_{\delta (G)}$, say $w$. Then, $v$ dominates all vertices in $K_{\delta (G)}$ and  $w$. Since every vertex $u$ in $K_{\delta (G)}$ has degree in $G$ equal to $\delta (G)$ and  $u$ has  $\delta (G)-1$ dominated neighbors, $u$ can force its one remaining neighbor. Therefore,  all vertices in $K_{\delta (G)}$ and their neighbors are observed. Since $\diam (G)=2$,   $d(v,x)=1$ or $d(v,x)=2$ for every vertex $x\ne v$ in $G$. If  $d(v,x)=1$, then $x$ is dominated by $v$. If $d(v,x)=2$, then $x$ is a neighbor of a vertex in $N_G(v)$. Since the vertices in $N_G(v)$ that are in $K_{\delta (G)}$ have forced their neighbors, the only case in which $x$ is not observed is if it is a neighbor of $w$.  Thus $\{v,w\}$ is a power dominating set.
\epf

\begin{cor}\label{pdle2}
Assume that $G$ and $\Gc$ both have all components of order at least $3$.   Then $\pd(G)\le 2$ or $\pd(\Gc)\le 2$ if any of the conditions below is satisfied:\vspace{-3pt}
\ben
\item\label{c2}  $\diam(G)\geq 3$ or $\diam(\Gc)\geq 3$.\vspace{-3pt}
\item\label{c1a} $G$ or $\Gc$ is planar.\vspace{-3pt}
\item\label{c3} $\kappa(G)\le 3$  or $\kappa(\Gc)\le 3$. \vspace{-3pt}
\item\label{c4} $G$ or $\Gc$ is not super-$\lambda$. \vspace{-3pt}
\een
\end{cor}
\bpf Part 
 \eqref{c2} follows from 
 Theorem \ref{diam3+}.  Since $G$ and $\Gc$ both have all components of order at least $3$, $\diam(G)\ne 1$ and $\diam(\Gc)\ne 1$.  The case $\diam(G)\ge 3$ is covered by part \eqref{c2}.  So assume  $\diam(G)=2$. Then \eqref{c1a}, \eqref{c3}, and \eqref{c4} follow from Corollary \ref{planar2pd}, Theorem \ref{kappapd}, and  Proposition \ref{superlambda}, respectively. 
\epf

Let $\maxpd$ be the family of graphs constructed by starting with a  connected graph $H$ and for each $v\in V(H)$ adding two new vertices $v'$ and $v''$, each adjacent to $v$ and possibly to each other but not to any other vertices.
The next result appears in \cite{ZKC06} without the floor function.

\begin{thm}{\rm \cite{ZKC06}}\label{n/3boundcor}
Suppose every component of a graph $G$ has order at least $3$ and $n$ denotes the order of $G$.  Then $\pd(G) \leq \lf\frac{n}{3}\rf.$  Furthermore, if $\pd(G)=\frac n 3$, then every component of $G$ is in $\maxpd\cup\{K_{3,3}\}$. \end{thm}

The method used in the construction of  a graph $G\in\maxpd$ implies that $\pd(\Gc)=1$ if we start with a graph on at least 2 vertices:

\begin{lem}\label{foliatecomplem} 
Suppose $G$ is a graph  having vertices $w,u, v$, $v'$, and $v''$ such that $N[v']=N[v'']= \{v, v',v''\}$, $u\in N(v)$ and $w\not\in N(v)$.  Then $\pd(\Gc) = 1$.  
\end{lem}
\begin{proof}
 In $\Gc$, $u$ is not adjacent to $v$ but is adjacent to $v'$ and to $v''$.  Then  $\{v'\}$ is  a power dominating set for $\Gc$, because $u \in N_{\overline{G}}[v']=V(G)\setminus\{v,v''\}$ and  $u$ forces $v''$ in $\Gc$, and then $w$  forces $v$ in $\Gc$.  Thus  $\gamma_{p}(\overline{G}) =1$.    \end{proof}

\begin{prop}\label{foliatecomp} 
Suppose $G$ is a graph of order $n$ such that every component of $G$ and $\Gc$ has order at least $3$ and $\pd(G)=\frac n 3$. Then $\pd(\Gc)\le 2$.  
If, in addition, $G$   has a component $G_o\in \maxpd$ of order at least $6$, then $\pd(\Gc) = 1$.
\end{prop}

\begin{proof}
Necessarily, $n$ is a multiple of 3 and $n\ne 3$.  If $G$ has 2 or more components, then $\gamma(\Gc)\le 2$ by Theorem \ref{diam3+}.  If $G=K_{3,3}$, then $\pd(\Gc)=2$.  Now suppose $G$ has a component $G_o\in \maxpd$ of order at least $6$ (this includes the case where $G$ has only one component that is not $K_{3,3}$).  Then $\pd(\overline{G_o})=1$ by Lemma \ref{foliatecomplem} and  Proposition \ref{GBarBd1} (for the case $v''\not\in N(v')$).  In $\Gc$, any vertex in $G_o$ dominates any vertex in a different component, so the one vertex that power dominates $\overline{G_o}$ also power dominates $\Gc$, and  $\pd(\Gc)=1$.  \end{proof}


\begin{thm}\label{domn4} {\rm\cite{HV06, MV14}}
Suppose $G$ is a graph of order $n$ with $\diam(G)=2$.  If $n\ge 24$, then $\gamma(G)\le \lf \frac n 4\rf$, and $\gamma(G)\le \lf \frac n 4\rf+1$ for $n\le 23$.
\end{thm}

\begin{rem}\label{twins}  Let $G$ be a graph.  Suppose $W$ is a set of at least two vertices such that no vertex outside $W$ is adjacent to exactly one vertex in  $W$.  
Then every power dominating set $S$ must contain either a neighbor of $W$ or a vertex in $W$, because no  vertex outside of $W$ can force a vertex in $W$ unless  all but one  of the vertices in $W$ have already been power dominated.
\end{rem}

For $r\ge 2$, the $r$th {\em necklace} graph, denoted by $N_r$, is constructed from $r$ copies of $K_4-e$ ($K_4$ with an edge deleted) by arranging them in a cycle and adding an edge between  vertices of degree $2$ in two consecutive copies of $K_4-e$.

\begin{thm}\label{3reg} {\rm\cite{{DHLMR13}}}
Suppose $G$ is a connected $3$-regular graph of order $n$ and $G\ne K_{3,3}$.  Then $\pd(G)\le \lf \frac n 4\rf$, and this bound is attained for arbitrarily large $n$ by $G=N_r$.
\end{thm}

\begin{lem}\label{Nrlem} For $r\ge 2$, $\pd(\overline{N_r})=2$.
\end{lem}  
\bpf Any two vertices that are in different copies of $K_4-e$ and are not incident to the missing edges dominate $\overline{N_r}$, so $\pd(\overline{N_r})\le 2$.
To complete the proof, we show that no one vertex $v$ can power dominate $\overline{N_r}$.  
Denote the vertices of the $K_4-e$ that contains $v$ by $x,y,z,w$, where $e=\{x,y\}$. 
Apply Remark \ref{twins} to $W=\{z,w\}$ for $v=x$ and to $W=\{x,y,w\}$ for $v=z$  to conclude $\{v\}$ is not a power dominating set; the cases $v=y$ or $w$ are similar. 
\epf


\section{Nordhaus-Gaddum sum bounds for power domination}\label{sNGsum}

In this section, we improve the tight Nordhaus-Gaddum sum upper bound of $n$ for all graphs (Corollary \ref{NGpowerdom}) to approximately $\frac n 3$  under one of  the assumptions that each component of $G$ and $\Gc$ has order at least 3 (Theorem \ref{NGsumthm} below), or that both $G$ and $\Gc$ are connected (Theorem \ref{NGsumthmcon} below), and to approximately $\frac n 4$ in some special cases.
The lower bound $2\le\pd(G)+\pd(\Gc)$ can be attained with both $G$ and $\Gc$  connected, specifically by the path $G=P_n$ (both $P_n$ and $\overline{P_n}$ are connected for $n\ge 4$).
But the upper bound for all graphs is attainable only by disconnecting $G$ or $\Gc$ with some very small components.  

The next result   follows from    Corollary \ref{pdle2}  and 
Theorem  \ref{n/3boundcor}.

\begin{cor}\label{diam3+pdsum}
Let $G$ be a  graph of order $n$ such that  every component of $G$ and $\Gc$ has order at least $3$ and $(\diam(G)\geq 3$ or $\diam(\Gc)\geq 3$ or $\kappa(G)\le 3$  or $\kappa(\Gc)\le 3)$.  Then $\pd(G) + \pd(\Gc) \leq  \lf \frac{n}{3}\rf+2$.
\end{cor}


\begin{thm}\label{NGsumthm}
Suppose $G$ is a graph of order $n$ such that every  component of $G$ and  $\Gc$ has order at least $3$.  Then for $n\ne  13, 14, 16, 17, 20$,  \[ \pd(G) + \pd(\Gc) \leq \lf\frac{n}{3}\rf + 2, \]
and this bound is attained for arbitrarily large $n$ by $G=rK_3$ $($where $r\ge 2)$.\\
For $n= 13, 14, 16, 17, 20$,
$ \pd(G) + \pd(\Gc) \leq \lf\frac{n}{3}\rf + 3. $
\end{thm}

\bpf
Without loss of generality, we assume $\pd(G)\le \pd(\Gc)$, and let $p=\pd(G)$ and $\pbar=\pd(\Gc)$. 
If $p\le 2$, then $p+\pbar \leq \lf\frac{n}{3}\rf + 2$ follows from Theorem \ref{n/3boundcor}.    If $p\geq 6$, Corollary \ref{cor:NGpowerdom} gives 
$ p+\pbar \leq \frac{n}{\pbar}+\frac{n}p  \leq \frac{n}{3}.$ 
So we assume  $3\le p\le 5$.  Since $\diam(G),\diam(\Gc)\ne 1$, by Corollary \ref{diam3+pdsum} we may also assume $\diam(G)=\diam(\Gc)=2$ and $\kappa(G),\kappa(\Gc)\ge 4$.  The latter implies $n\ge 9$.
  Corollary \ref{cor:NGpowerdom} implies
$p+\pbar \leq p + \lf\frac{n}{p}\rf$. By Theorem \ref{domn4}, $p,\pbar \le\lf\frac n 4\rf+1$.  Thus we need to consider the following cases:\vspace{-3pt}
\bit
\item $p=3, 4$, in which case $p+\pbar\le  \lf\frac n 4\rf+4$.\vspace{-3pt}
\item $p=5$, in which case $p+\pbar  \le\lf\frac n 5\rf+5$.\vspace{-3pt}
\eit Algebra shows that  $\lf\frac n 4\rf+4 \le \lf \frac n 3 \rf +2$ and $\lf\frac n 5\rf+5 \le \lf \frac n 3 \rf +2$ for  $n\ge 21$ and $n=18, 19$.
  For $n=9, 10, 11$,  $p+\pbar\le  5=\lf\frac n 3\rf+2$  has been verified computationally \cite{sagedata}.

To complete the proof that $p+\pbar \leq  \frac{n}{3}+2$ for $n\ne  13, 14, 16, 17, 20$, we consider $n=12, 15$.  Since $p\le \pbar\le \frac n p$, the only possibilities are $n=12$ with $(p,\pbar)=(3,3), (3,4)$, or $n=15$ with $(p,\pbar)=(3,3), (3,4), (3,5)$. For $n=12$ with $ (p,\pbar)=(3,3)$, and $n=15$ with $(p,\pbar)=(3,3), (3,4)$,  $\pd(G)+ \pd(\Gc) \le \frac{n}{3}+2$.  In each of the remaining  cases, $n=12$ with $ (p,\pbar)=(3,4),$ or $n=15$ with $ (p,\pbar)= (3,5)$, observe that $\pbar=\frac n 3$ and $p=3$.  But this is prohibited by Proposition \ref{foliatecomp}.  

 If $G$ is a disjoint union of $r\ge 2$ copies of $K_3$, then  $\pd(G)+ \pd(\Gc) = \frac{n}{3}+2$, so the bound is tight for arbitrarily large $n$.
 
Finally, consider $n=  13, 14, 16, 17, 20$.  For $p=3$, $ \pd(G) + \pd(\Gc) \leq \lf\frac{n}{3}\rf + 3 $ is immediate from Theorem \ref{n/3boundcor}.  Since $p\le \pbar\le \frac n p$, the only remaining cases are $n=16$ or 17 with  $(p,\pbar)=(4,4)$, and $n=20$ with $(p,\pbar)=(4,4),\, (4,5)$. All of these satisfy
$ \pd(G) + \pd(\Gc) \leq \lf\frac{n}{3}\rf + 3 $.   \epf
 
We have no examples  contradicting  $\pd(G)+\pd(\Gc)\le  \lf\frac{n}{3}\rf+2$ for graphs $G$ of any order $n$ where  the order of each component of $G$ and $\Gc$ is at least 3.  We conjecture that these ``exceptional values"  $13, 14, 16, 17, 20$ of $n$ are not in fact exceptions:

\begin{conj}\label{sumupper3} If $G$ is graph of order $n$ such that   the order of each component of $G$ and $\Gc$ is at least $3$, then $\pd(G)+\pd(\Gc)\le  \lf\frac{n}{3}\rf+2$.
\end{conj} 

Next we consider the case in which both $G$ and $\Gc$ are required to be connected.

\begin{thm}\label{NGsumthmcon}
Suppose $G$ is a graph of order $n$ such that both $G$ and  $\Gc$ are connected.  Then for $n\ne 12, 13, 14, 15, 16, 17, 18, 20, 21, 24$,  \[ \pd(G) + \pd(\Gc) \leq \lc\frac{n}{3}\rc + 1, \]
and this bound is attained for arbitrarily large $n\ge 6$ by $G\in\maxpd$.
\end{thm}
 \bpf For $n$ not a multiple of 3, $\lc\frac{n}{3}\rc + 1=\lf\frac{n}{3}\rf + 2$, and the result follows from Theorem \ref{NGsumthm}.   So assume  $n$  is a multiple of 3.  We proceed as in the proof of Theorem \ref{NGsumthm}, with the same  notational conventions $p:=\pd(G)\le \pbar:=\pd(\Gc)$, and again the bound is established for $p\ge 6$.  If $p\le 2$, then $p+\pbar \leq \frac{n}{3} + 2$ follows from Theorem \ref{n/3boundcor}, and the only way to attain $p+\pbar = \frac{n}{3} + 2$ is to have $p=2$ and    $\pbar = \frac{n}{3}$.  Since $\overline{K_{3,3}}$ is not connected, Theorem \ref{n/3boundcor} and 
 Proposition \ref{foliatecomp}  prohibit $p\ge 2$ and    $\pbar = \frac{n}{3}$.
So we assume  $p\le 5$ and  $3\le p\le \pbar\le \frac n 3 -1$; the latter requires $n\ge 12$.   Algebra shows that  $\lf\frac n 4\rf+4 \le  \frac n 3  +1$ and $\lf\frac n 5\rf+5 \le \frac n 3   +1$ for  $n\ge 27$.

  There are  graphs $G\in\maxpd$ of arbitrarily large order $n$, $\Gc$ is connected for $n\ge 6$, and these graphs attain the bound. 
\epf

The tight upper bound in Theorem \ref{NGsumthmcon} for $\pd(G)+\pd(\Gc)$  with both $G$ and $\Gc$ connected  was obtained by switching from floor to ceiling.  This raises a question about the bound with floor, which has implications for products (see Section \ref{sNGprod}).

\begin{quest}\label{NGsumconQ} Do there exist graphs  $G$  of arbitrarily large order $n$ with both $G$ and $\Gc$ connected such that $\pd(G) + \pd(\Gc) =\lf \frac{n}{3}\rf+2$?
\end{quest}

 The next two examples, found via the computer program {\em Sage}, show  that there are pairs of connected graphs  $G$ and $\Gc$ of orders $n=8$ and $11$ such that $\pd(G)+\pd(\Gc)= \lf\frac{n}{3}\rf+2$.  

\begin{ex}\label{cx:n/3+1} Let $G$  be the graph shown with  its complement  in Figure \ref{fig:n/3+1}; observe that both are connected.  It is easy to see that no one vertex power dominates either $G$ or $\Gc$ and also easy to find a power dominating set of two vertices for each. 
Thus  
\[\pd(G)+\pd(\Gc)=2+2=4=\lf\frac 8 3\rf+2.\]

\begin{figure}[!ht]
\begin{center}
\scalebox{.5}{\includegraphics{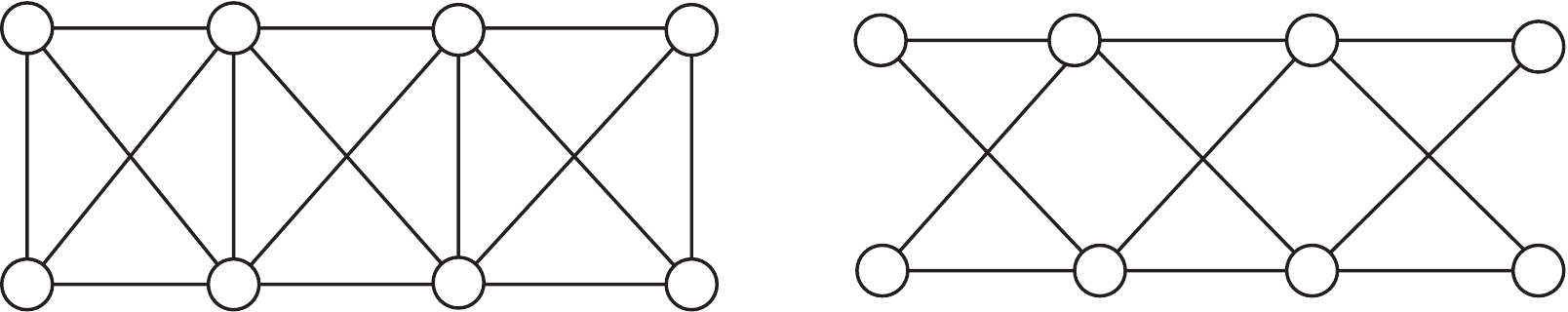}}\vspace{-5pt}
\caption{A connected graph $G$ of order 8 and its connected complement $\Gc$ such that $\pd(G)+\pd(\Gc)=\lf\frac{n}{3}\rf+2.$}\label{fig:n/3+1}\vspace{-8pt}
 \end{center}
\end{figure}
\end{ex}

\begin{ex}\label{newcx} Let $G$ be the graph  shown in Figure \ref{fig:n/3+2}.  It is easy to see that $\Gc$ is also  connected. 

\begin{figure}[!h]
\begin{center}
\scalebox{.7}{\includegraphics{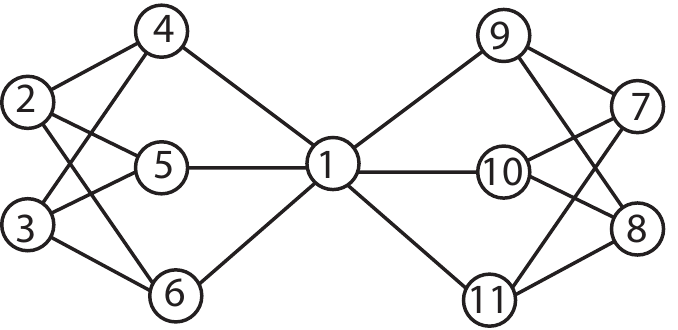}}\vspace{-5pt}
\caption{A connected graph $G$ of order 11 such that $\Gc$ is also connected and  $\pd(G)+\pd(\Gc)=\lf\frac{n}{3}\rf+2.$}\label{fig:n/3+2}\vspace{-15pt}
 \end{center}
\end{figure}

First we show that no set of two vertices is a power dominating set for $G$.  Since $\{1,2,7\}$ is a power dominating set for $G$, this will imply $\pd(G)=3=\lf\frac {11}3\rf$.  By Remark \ref{twins} applied to the sets $W_1=\{2,3\}$ and $W_2=\{7,8\}$, any power dominating set $S$ of $G$ must contain vertices $u_1\in\{ 2, 3, 4, 5, 6\}$ and analogously,  $u_2\in\{ 7, 8, 9, 10,  11\}$.  If $u_1\in\{2, 3\}$ and  $u_2\in\{7,8\}$, then vertex $1$ cannot be forced. If $u_1\in\{4,5,6\}$, then the two remaining vertices in $\{4,5,6\}$ cannot be forced; the case in which  $u_2\in\{9,10,11\}$ is symmetric. 

Next we show that no one vertex is a power dominating set for $\Gc$.  Since $\{2,7\}$ is a power dominating set for $\Gc$, this will imply $\pd(\Gc)=2$ and  $\pd(G)+\pd(\Gc)=\lf\frac {11}3\rf+2$. For each possible vertex $v\in \{1,2,3,4,5,6\}$, we apply Remark \ref{twins} with $W$ as shown: For $v\in\{1,2,3\}$, use $W=\{4,5,6\}$. For $v\in \{4,5,6\}$, use $W=\{2,3\}$. The case $v\in \{7,8,9,10,11\}$ is symmetric.
\end{ex}

The next two theorems for domination number provide an interesting comparison.


\begin{thm}{\rm \cite{AJ95}}\label{NGdomJA95}
For any graph $G$ of order $n$ such that $\delta(G)\ge 1$ and $\delta(\Gc)\ge 1$,
\[ \gamma(G) +\gamma(\Gc)\le \lf \frac {n}{2}\rf+2,\]
and this bound is attained for arbitrarily large $n$.  \end{thm}

\begin{thm}\label{n3dom}
{\rm \cite{HV06}} Suppose $G$ is a graph of order $n$ such that $\delta(G)\ge 7$ and $\delta(\Gc)\ge 7$.  Then
 \[ \gamma(G) + \gamma(\Gc) \leq \lf\frac{n}{3}\rf + 2. \]
\end{thm}

From Theorem \ref{n3dom} we see that the same sum upper bound we obtained for power domination number (with the weaker hypothesis that every component has order at least 3) is obtained for domination number when we make the stronger assumption that the minimum degrees of both $G$ and $\Gc$ are  at least 7.
Theorem \ref{NGdomJA95} is a more direct parallel to Theorem \ref{NGsumthm} but with a higher bound.    Theorem \ref{NGdomJA95} has a weaker hypothesis, which is equivalent to ``every component of  $G$ and $\Gc$ has order at least 2."   The next example shows that if Theorem \ref{NGdomJA95} is restated to require  
 both $G$ and $\Gc$ to be connected, the bound remains tight.  This provides a direct comparison with Theorem  \ref{NGsumthmcon} and shows that  for graphs $G$ with both $G$ and $\Gc$  connected, the upper bound  for the domination sum is substantially higher than the upper bound for the power domination sum.

\begin{ex}\label{ex:comb}  Let $G_k$ denote the $k$th comb, constructed by adding a leaf to every vertex of a path $P_k$ ($G_9$ is shown in Figure \ref{fig:comb}); the order of $G_k$ is $2k$.  Then every dominating set $S$ must have at least $k$ elements, because for each of the $k$ leaves, either the leaf or its neighbor must be in $S$. Since two vertices are needed to dominate $\overline{G_k}$,  $\gamma(G_k)+\gamma(\overline{G_k})=k+2=\frac{2k} 2 +2$.  The results for power domination are very different.  For $k=3s$,  one third of the vertices in $P_k$ can power dominate $G_k$, and one vertex can power dominate $\overline{G_k}$, so 
$\pd(G_k)+\pd(\overline{G_k})=s+1=\frac{2k} 6 +1$.

\begin{figure}[!ht]
\begin{center}
\scalebox{.7}{\includegraphics{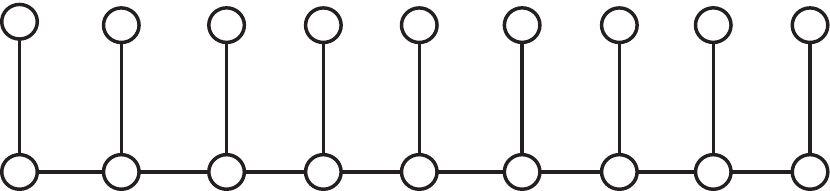}}\vspace{-5pt}
\caption{The comb $G_9$ with the vertices of a minimum power dominating set colored. 
}\label{fig:comb}\vspace{-8pt}
 \end{center}
\end{figure}
\end{ex}

We can also 
improve the bound in Corollary \ref{NGpowerdom} when $G$ has some components of order less than 3 and $G$ has at least one edge.

\begin{thm}\label{smallcomp}
Let $G$ be a graph of order $n$ that has $n_1$ isolated vertices and $n_2$ copies of $K_2$ as components such that   $n_1$ and $n_2$ are not both zero. Then\vspace{-3pt}
 \[ \pd(G) + \pd(\Gc) \leq 1 + \frac{n}{3} + \frac{2n_1}{3} + \frac{n_2}{3}. \vspace{-3pt}\]
\end{thm}

\bpf
As a consequence of Theorem \ref{n/3boundcor},\vspace{-3pt}
 \[ \pd(G) \leq n_1 + n_2 + \left( \frac{n - n_1 - 2n_2}{3} \right).\vspace{-3pt} \]  Because $n_1\geq 1$ or $n_2 \geq 1$, an isolated vertex (respectively, one of the vertices in a $K_2$ component) power dominates the complement, so $\pd(\Gc) = 1$.  Hence,\vspace{-3pt}
\[ \pd(G) + \pd(\Gc) \leq  n_1 + n_2 + \left( \frac{n - n_1 - 2n_2}{3} \right) + 1. \qedhere\]
\epf

We can also improve the upper bound in some special cases.  

\begin{thm}\label{n4sum}
Suppose $G$ is a graph of order $n$ with $\diam (G)=\diam(\Gc)=2$, and one of the following is true: \vspace{-3pt}
\ben
\item $G$ or $\Gc$ is planar.\vspace{-3pt}
\item $\kappa(G)\le 3$  or $\kappa(\Gc)\le 3$. \vspace{-3pt}
\item $G$ or $\Gc$ is not super-$\lambda$.\vspace{-3pt}
\een
If $n\ge 24$, then $\pd (G)+ \pd (\Gc)\leq \lf {n\over 4}\rf+2$, and 
  $\pd (G)+ \pd (\Gc)\leq \lf {n\over 4}\rf+3$ for $n\le 23$. 
\end{thm}
\begin{proof}
By Corollary \ref{pdle2}, $\pd (G)\leq 2$ or $\pd (\Gc)\leq 2$. 
Assume without loss of generality that $\pd (G)\leq 2$. Applying Theorem \ref{domn4} to $\Gc$,   
$\pd (\Gc)\leq \lf {n\over 4}\rf$ for $n\ge 24$ and $ \pd (\Gc)\leq \lf {n\over 4}\rf+1$ for $n\le 23$. 
\end{proof}


\begin{thm}\label{3regNGsum}
Suppose $G$ is a $3$-regular graph of order $n\ge 6$ such that  no component  is $ K_{3,3}$.  Then $\pd(G)\le \lf \frac{n}{4} \rf$,  $\pd(\Gc) \leq 2$, and  $\pd(G) + \pd(\Gc) \leq \lf \frac{n}{4} \rf + 2$, and all these inequalities are tight for arbitrarily large $n$.
\end{thm}

 \bpf
Suppose first that $G$ is connected.  Then $\pd(G)\le\lf\frac n 4 \rf$ by Theorem \ref{3reg} (since  $G\neq K_{3,3}$), so it suffices to show $\pd(\Gc)\le 2$. Since $G\ne K_4$ and $G$ is 3-regular, $\diam(G)\ge 2$.  Since $\diam(G)\ge 3$ implies $\pd(\Gc)\le 2$ by Theorem \ref{diam3+}, we  assume $\diam(G)=2$.  For any vertex $v$, there are at most 10 vertices at distance 0, 1, or 2 from $v$ ($v$, its 3 neighbors, and  two additional neighbors of each of the neighbors of $v$), so  $n\le 10$.  An examination of  3-regular graphs with $6\le n \le 10$  (see, for example, \cite[p. 127]{atlas}) shows the only such graphs  of diameter 2 are the five graphs shown in Figure \ref{fig:cubic} (named as in \cite{atlas}): C3 = $K_{3,3}$, C2, C5, C7, and C27 (the Petersen graph). It is straightforward to verify that $\pd(\Gc)=1$ for $G\in\{$C2, C7$\}$ and $\pd(\Gc)=2$ for $G\in\{$C5, C27$\}$.  
This completes the proof for the case in which $G$ is connected.

\begin{figure}[!ht]
\begin{center}
\scalebox{.5}{\includegraphics{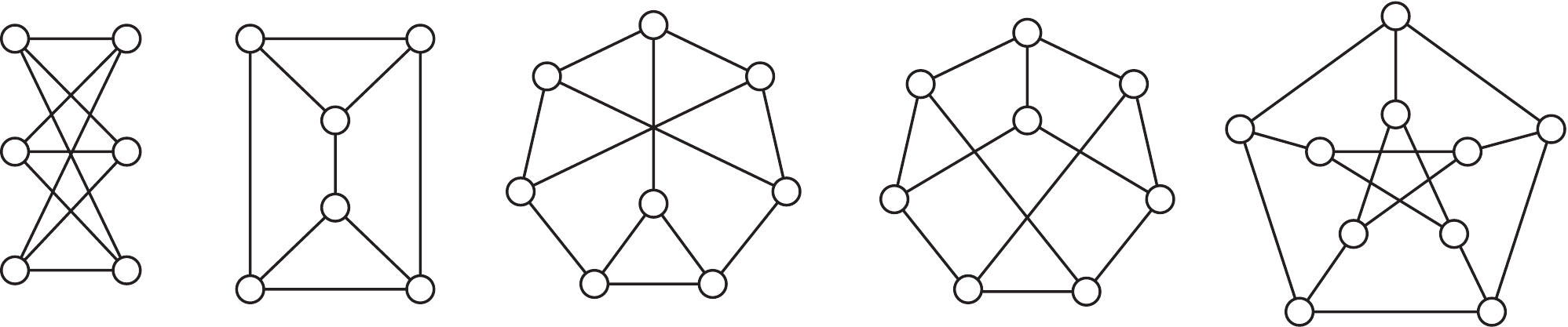}}\vspace{-5pt}
\caption{The five cubic graphs of diameter 2: C3 = $K_{3,3}$, C2, C5, C7, and C27 = the Petersen graph. }\label{fig:cubic}\vspace{-8pt}
 \end{center}
\end{figure}

Now assume $G$ has components $G_1,\dots,G_s$ with $s\ge 2$.  Then $\pd(\Gc)\le 2$ by Theorem \ref{diam3+}.  Since $G_i\ne K_{3,3}$, \vspace{-3pt}
\[\pd(G)=\sum_{i=1}^s \pd(G_i)\le \sum_{i=1}^s\lf \frac{n_i}4\rf\le \lf \frac{\sum_{i=1}^sn_i}4\rf=\lf \frac n 4\rf.\]

The graphs $N_r$ attain the bound by Theorem \ref{3reg} and Lemma \ref{Nrlem}.
\epf


\section{Nordhaus-Gaddum product bounds for power domination}\label{sNGprod}

As with the sum, the tight product lower bound for the power domination number for all graphs $G$ remains unchanged even with the additional requirement that both $G$ and $\Gc$ be  connected (using the path).  In Section \ref{sNGsum}, we achieved a tight sum upper bound for such graphs. However, since this was achieved with $\pd(\Gc)=1$ for both $G$ and $\Gc$ connected, and with $\pd(\Gc)=2$ when each component of both $G$ and $\Gc$ has order at least 3, there are few immediate implications for products (see  Section \ref{sdiscuss} for further discussion of connections between sum and product bounds).  

\begin{quest}\label{ordge3prod} Does there exist a graph $G$  of order $n$ such that all components of $G$ and $\Gc$ have order at least $3$ and $\pd(G) \cdot \pd(\Gc) > 2\lf \frac{n}{3}\rf$?
\end{quest}

\begin{rem} \label{rK3p} If  the answer to Question \ref{ordge3prod} is negative, then the graphs $G=rK_3$ with $r\ge 2$ 
show $2\lf \frac{n}{3}\rf$ is a tight  upper bound for the product, because $\pd(G)=\frac n 3$ and $\pd(\Gc)=2$.
\end{rem}

\begin{rem} \label{conprodrem} If  the answer to Question \ref{NGsumconQ} is positive, then such graphs 
show $2\lf \frac{n}{3}\rf$ can be attained for arbitrarily large $n$ for the product with both $G$ and $\Gc$ connected.
\end{rem}

We can improve the product bound in certain special cases. 
The next result follows from  Corollary \ref{pdle2}  and Theorem \ref{n/3boundcor}.

\begin{cor}
Let $G$ be a  graph of order $n$ such that  every component of $G$ and $\Gc$ has order at least $3$.   Then   $\pd(G) \cdot \pd(\Gc) \leq 2\lf \frac{n}{3}\rf$ if at least one of the following is true:\vspace{-3pt}
\ben
\item  $\diam(G)\geq 3$ or $\diam(\Gc)\geq 3$.\vspace{-3pt}
\item $G$ or $\Gc$ is planar.\vspace{-3pt}
\item $\kappa(G)\le 3$  or $\kappa(\Gc)\le 3$. \vspace{-3pt}
\item $G$ or $\Gc$ is not super-$\lambda$.\vspace{-3pt}

\een

 \end{cor}




The next two results are  product analogs of Theorems \ref{smallcomp} and  \ref{n4sum}.  The proofs, which are analogous, are omitted.

\begin{thm}
Let $G$ be a graph of order $n$ that has $n_1$ isolated vertices and $n_2$ copies of $K_2$ as components such that   $n_1$ and $n_2$ are not both zero.  Then\vspace{-3pt}
 \[ \pd(G) \cdot \pd(\Gc) \leq \frac{n}{3} + \frac{2n_1}{3} + \frac{n_2}{3}.\vspace{-3pt} \]
\end{thm}

\begin{thm}
Suppose $G$ is a graph of order $n$ with $\diam (G)=\diam(\Gc)=2$, and one of the following is true: \vspace{-3pt}
\ben
\item $G$ or $\Gc$ is planar.\vspace{-3pt}
\item $\kappa(G)\le 3$  or $\kappa(\Gc)\le 3$. \vspace{-3pt}
\item $G$ or $\Gc$ is not super-$\lambda$.\vspace{-3pt}
\een
If $n\ge 24$, then $\pd (G)\cdot \pd (\Gc)\leq 2\lf {n\over 4}\rf$, and $ \pd(G)\cdot \pd (\Gc)\leq 2\lf {n\over 4}\rf+2$ for $n\le 23$. 
\end{thm}

The next result follows immediately from Theorem  \ref{3regNGsum}.
\begin{cor}
Suppose $G$ is a $3$-regular graph of order $n\ge 6$ with no $ K_{3,3}$ component.  Then $\pd(G) \cdot \pd(\Gc) \leq 2\lf \frac{n}{4} \rf $, and this bound is attained for arbitrarily large $n$.
\end{cor}


\begin{prop}
Let $G$ be a tree on $n \ge 4$ vertices. If $G$ is not $K_{1,3}$ or $K_{1,4}$, then\vspace{-3pt}
 \[\pd(G) \cdot\pd(\Gc)\le \lf\frac{n}{3}\rf \vspace{-3pt}\] and this bound is attained for arbitrarily large $n$.
\end{prop}

\bpf Note first that since $G$ is connected, $\pd(G) \le \lf\frac{n}{3}\rf$ by Theorem \ref{n/3boundcor}.   If a tree is not   a star, then its complement is also connected, and by Proposition \ref{GBarBd1}, $\pd(\Gc) = 1$.    For a star graph $K_{1,n-1}$, we have $\pd(K_{1,n-1})\cdot\pd(\overline{K_{1,n-1}})=2$, which is less than or equal to $\frac{n}{3}$ when $n \ge 6$. 
 The bound is attained for arbitrarily large $n$ because if $G$ is constructed from any tree $T$    by adding two leaves to each vertex of $T$, then $\pd(G)=\frac n 3$.
\epf

\section{Summary and discussion}\label{sdiscuss}

Table \ref{sumtab} summarizes what is known about Nordhaus-Gaddum sum bounds for power domination number, domination number, and zero forcing number.  

\begin{table}[h!]\caption{Summary of tight bounds for $\zeta(G)+\zeta(\Gc)$ for $\zeta = \pd, \gamma, \Z$}\label{sumtab}
\begin{center} 
\renewcommand{\arraystretch}{1.3}
\noindent \begin{tabular}{|c|c|c|}
\hline 
$\zeta$ \& restrictions &  lower &  upper \\
\hline 
\hline 
$\pd$ & $2$ &  $n+1$\\
\hline
$\pd$ \& all components  of both $G$ and $\Gc$ of order $\ge 3$ \& $n\ge  21$  & $2$   & $\lf \frac n 3 \rf +2$\\ 
\hline 
$\pd$ \& both $G$ and $\Gc$ connected \& $n\ge  25$  & $2$   &  $\lc \frac n 3 \rc +1$\\ 
\hline
\hline 
$\gamma$ \& $n\ge 2$ & $3$   & $n+1$\\
\hline
$\gamma$ \& both $G$ and $\Gc$ connected 
& $3$   & $\lf \frac n 2 \rf +2$\\[.4mm]
\hline 
\hline 
$\Z$ \& $n\ge 2$ & $n-2$  & $2n-1$\\
\hline $\Z$ \& both connected  & $n-2$  & $ 2n-o(n)$\\
\hline 
\hline 
\end{tabular}\vspace{-8pt}
\end{center}
\end{table}

Both the sum and product upper and lower bounds for the domination number were determined by Jaeger and Payan in 1972 (see Theorem \ref{NGdom}), and analogous bounds for power domination are immediate corollaries.  Since then, there have been numerous improvements to the sum upper bound for domination number under various conditions on $G$ and $\Gc$.  Examples of such conditions include requiring every component of both   $G$ and $\Gc$ to have order at least 2 or requiring both to be connected 
or requiring both to have minimum degree at least  7. 
In Section \ref{sNGsum} we established   better upper bounds for the power domination number in the cases where both $G$ and $\Gc$ are connected or both  have every component of order at least 3.   

 By contrast, results on products are very sparse for both domination number and power domination number.  Historically, the Nordhaus-Gaddum sum upper bound has often been determined first, and then used to obtain the product upper bound, as in 
 the case of Nordhaus and Gaddum's original results \cite{NG} (see Theorem  \ref{NGthm}).  In order to use this technique of getting a tight product bound from a tight sum bound, one needs the sum upper bound to be optimized with approximately equal values or the sum lower bound to be optimized on extreme values. The sum lower bound for the domination number  is optimized at the extreme values, and therefore the tight lower bound for the sum yields a tight lower bound for the product.  However, all available evidence suggests that, for both the domination number and  the power domination number, the sum upper bound is optimized only  at extreme values.  For example,  the sum upper bound of $n+1$ over all graphs  is attained only by the values $1$ and $n$ for both the domination and power domination numbers. 
  Thus, for the domination number and the power domination number, the Nordhaus-Gaddum product  upper bound  presents  challenges.

Further evidence indicating that the sum bound is optimized only on extreme values comes from random graphs.  
And  it is also interesting to consider the ``average'" behavior, or expected value, of the sum and product of $\Z, \gamma$, and $\pd$ using  the Erd\H{o}s R\'enyi random graph $G(n,\frac 1 2)$ (whose complement is also a random graph with edge probability $\frac 1 2$).  
Let $G=G(n,\frac 1 2)$.  Then $\Z(G)=n-o(n)$, since  $\tw(G)=n-o(n)$ \cite{PS14} and $\tw(H)\le \Z(H)\le n$ for all graphs of order $n$ ($\tw(H)$ denotes the tree-width of $H$).  Thus $\Z(G)+\Z(\Gc)= 2n-o(n)$ and $\Z(G)\cdot\Z(\Gc)= n^2-o(n^2)$, and this establishes that the upper bound listed in Table \ref{sumtab} for  connected graphs $G$ and $\Gc$.
 For any $\epsilon>0$, $(1-\epsilon)\log_2n\le\gamma(G)\le (1+\epsilon)\log_2n$ with probability going to 1 as $n\to\infty$ \cite{NS94, WG01}.  Thus    $\gamma(G)+\gamma(\Gc)=2\log_2n\pm o(\log_2 n)$ and  $\gamma(G)\cdot\gamma(\Gc)=(\log_2n\pm o(\log_2 n))^2$ for $G=G(n,\frac 1 2)$.
Since $\pd(H)\le\gamma(H)$ for all graphs $H$, $\pd(G)\le \log_2n+o(\log_2 n)<<\lc \frac n 3 \rc +1$  for $G=G(n,\frac 1 2)$ as $n\to\infty$ (observe that $G$ and $\Gc$ are both connected with probability approaching 1 as $n\to \infty$).



\subsection*{Acknowledgements}
This research was supported by the American Institute of Mathematics (AIM), the  Institute for Computational and Experimental Research in  Mathematics (ICERM), and the National Science Foundation (NSF) through DMS-1239280, and the authors thank AIM, ICERM, and NSF. We also thank  S. Arumugam for sharing a paper with us via email that was very helpful, and  Brian Wissman for fruitful  discussions  about preliminary work and for providing the {\em Sage} code for power domination.







\end{document}